\documentclass[11pt,reqno]{amsart}
\usepackage[utf8]{inputenc}
\usepackage{amsmath, amsthm}
\usepackage{amssymb}
\usepackage[margin=1in]{geometry}
\usepackage[mathscr]{euscript}
\usepackage{titlesec}
\usepackage{braket}
\usepackage{shortcutpackagev11}
\usepackage{arydshln}
\usepackage{verbatim}
\usepackage{mathtools}
\usepackage{graphicx}
\usepackage[all]{xypic}
\graphicspath{ {./images/} }
\usepackage{url}
\usepackage{quiver}
\usepackage{parskip}
\usepackage[T1]{fontenc}
\usepackage{hyperref}
\usepackage{thm-restate}

\usepackage[
  backend=biber,
  style=alphabetic,
  sorting=nyt,
  maxbibnames=99,
  uniquename=true
]{biblatex}
\makeatletter
\renewcommand{\section}{\@startsection{section}{1}%
  \z@{.7\linespacing\@plus\linespacing}{.5\linespacing}%
  {\normalfont\bfseries\centering}}
\makeatother

\numberwithin{equation}{section}

\title{The Bias Conjecture for One-Parameter, Non-Isotrivial Elliptic Curve Families}

\author{Lucas Chen}
\address{Department of Mathematics, University of Chicago}
\email{\href{mailto:}{lucasch@uchicago.edu}}

\author{Joshua Im}
\address{Department of Mathematics, Texas A\&M University}
\email{\href{mailto:}{jri@tamu.edu}}

\author{Steven J. Miller}
\address{Department of Mathematics, Williams College}
\email{\href{mailto:}{sjm1@williams.edu}}

\author{Devayani Pradhan}
\address{Department of Mathematics, University of Michigan}
\email{\href{mailto:}{pradhand@umich.edu}}
\thanks{Corresponding Author: Devayani Pradhan}

\date{}

\begin{document}

\begin{abstract}
    We investigate Miller's bias conjecture for general one-parameter families of elliptic curves over $\mathbb{Q}$ with non-constant $j$-invariants. Utilizing methods developed by Michel, we view the second moment as the sum of distinct cohomological components: a main $p^2$ term, a $p^{3/2}$ term arising from the first cohomology over the projective line, and terms of order $p$ or lower corresponding to singular fibers. By assuming the generalized Sato-Tate conjecture, we employ representation and motive theory to demonstrate that the coefficient of the $p^{3/2}$ term averages to zero in the limit. We then independently analyze the lower-order contributions without assuming Sato-Tate. By applying the Chebotarev density theorem and analyzing Galois representations, we show that the contributions to the order $p$ term from additive singular fibers average to a non-positive value. Consequently, the only remaining average contribution to the $p$ term is $-B_pp$, where $B_p$ denotes the number of ``bad'' points where the discriminant vanishes. Because $B_p$ averages to a strictly positive value, the overall coefficient of the $p$ term averages to a strictly negative value in the limit, thereby establishing the bias conjecture for these families.
\end{abstract}

\maketitle

\newtheorem{theorem}{Theorem}[section]
\newtheorem{proposition}[theorem]{Proposition}
\newtheorem{lemma}[theorem]{Lemma}
\newtheorem{corollary}[theorem]{Corollary}

\theoremstyle{definition}
\newtheorem{definition}[theorem]{Definition}
\newtheorem{example}[theorem]{Example}

\theoremstyle{remark}
\newtheorem{conjecture}[theorem]{Conjecture}
\newtheorem{remark}[theorem]{Remark}

\setlength{\parindent}{0pt}


\section{Introduction}
The arithmetic of elliptic curves over global fields is deeply intertwined with their local behavior at finite primes. A central object of study in this context is a one-parameter family of elliptic curves, typically given by a Weierstrass equation of the form $\mathcal{E}: y^2 = x^3 + A(T)x + B(T)$, where $A(T), B(T) \in \mathbb{Z}[T]$. For a specific specialization parameter $t \in \mathbb{Z}$ and a prime $p$, the local point count of the reduced curve $E_t$ over the finite field $\mathbb{F}_p$ is characterized by the trace of Frobenius, defined as $a_t(p) = p + 1 - \#E_t(\mathbb{F}_p)$. The Birch and Swinnerton-Dyer (BSD) conjecture highlights the importance of analyzing the statistical distribution of these point counts to expose broader arithmetic structure.

The distribution of these local quantities plays a crucial role in understanding the analytic properties of the associated Hasse-Weil $L$-functions. Specifically, the Katz-Sarnak density conjecture \cite{KatzSarnak1999} models the distribution of the low-lying zeros (or values) of these families of $L$-functions using Random Matrix Theory. We compute this distribution by determining the moments of the $L$-function terms across the family. In this framework, the first moment of the $a_t(p)$ traces relates directly to the rank of the family, while the main term of the second moment is universal. Crucially, the lower-order corrections in the second moment have a direct impact on the fine-scale distribution of the zeros. (Separately, for typical non-CM elliptic curves, the limiting behavior of the individual normalized traces follows the Sato-Tate semicircular distribution \cite{Taylor2008}.)

The modern significance of $a_t(p)$ is elegantly captured through a simple corollary of the standard Grothendieck-Lefschetz trace formula.\footnote{Michel's argument rests on a cohomological generalization of this quantity, though he does not explicitly detail the one-curve case; we do so here to more easily motivate this machinery-heavy picture for the computational side of the field.}\begin{theorem}For $E_t$ a complete curve nonsingular over $\mathbb{F}_p$ and $\mathcal{F}$ a constructible $\mathbb{Q}_\ell$-sheaf, we have$$\sum_{x\in E_t(\mathbb{F}_p)} \mathrm{tr}(\mathrm{Frob}_{p,x}\mid \mathcal{F}_x)\ =\ \sum_{i=0}^2 (-1)^i\mathrm{tr}(\mathrm{Frob}_p\mid H^i_c(E_t, \mathcal{F}))$$\end{theorem}When we take $\mathcal{F}=\mathbb{Q}_\ell$ itself, the left side resolves to $\#E_t(\mathbb{F}_p)$.Explicitly computing the other two traces and noting that $E_t$ is proper (which equates $H_c$ and $H_\mathrm{et}$) yields our formula$$\text{tr}(\text{Frob}_p\mid H^1_\text{et}(E_t, \mathbb{Q}_\ell))\ =\ a_t(p)\ =\ p+1-\#E_t(\F_p).$$The classical Hasse bound, where the trace of Frobenius originally appeared, is easily recovered by applying Weil's Riemann hypothesis to the Frobenius eigenvalue summands: $\vert{}\alpha_1\vert{}=\vert{}\alpha_2\vert{}=\sqrt q$.

The structural clarity of bounding through cohomology motivates our view of $a_t(p)$ through this lens; this follows the philosophy of \cite{Michel1995}. We apply this viewpoint to the problem of \textit{biases} in Frobenius traces over an elliptic curve family, which we now detail.

It was already known that the limiting behavior of the first moment of $a_t(p)$ exhibits a bias, one intrinsically linked to the structural properties of the family—namely, its rank—via the Rosen-Silverman theorem \cite{Rosen1998, Nagao1997}. Building on this foundation, Steven J. Miller's thesis \cite{Miller2002} investigated the \textit{second} moment. While the main-order distribution of the $L$-function zeros is modeled by the Katz-Sarnak framework, we can compute that a lower-order term in the second moment of the traces perturbs the convergence rate. Miller formulated the \textit{bias conjecture} concerning this second-moment perturbation, investigating the arithmetic effect of such a bias on the 1- and 2-level densities of the family's associated Hasse-Weil $L$-functions.

This conjectured bias has been extensively studied in the following form for a one-parameter elliptic curve family $E_t$ with non-constant $j(T)$-invariant. \begin{conjecture}In the second moment expansion of $E_t$, the largest lower-order term which does not average to zero is, on average, negative.\end{conjecture}A range of computational literature has been published verifying or questioning this conjecture, though theoretical explanations have only recently surfaced. Significant tests of the bias can be found in \cite{Asada2023, batterman2024applications, Cheek2026}. While numerical studies often focus on standard one-parameter families over the rationals, recent generalizations have shown that the phenomenon is highly dependent on the setting. For instance, when analyzing elliptic curves over finite fields whose traces of Frobenius lie in fixed arithmetic progressions, the bias is not universally negative, and can actually be strictly positive for a positive density of arithmetic progressions \cite{kane2025bias}. We prove a sharpened statement of the bias conjecture via the cohomological viewpoint, previously examined with different methods by \cite{KN} for cubic pencils, and suggested in \cite{milneLEC}. Our proof suggests explanations for much of the computational literature.

Note that our proof is conditional on the Sato-Tate conjecture. The classical Sato-Tate conjecture predicts the statistical distribution of Frobenius traces for non-CM elliptic curves reduced over finite fields. This was proven for elliptic curves over totally real fields \cite{clozel2008automorphy}, and later generalized to all non-CM regular algebraic cuspidal automorphic representations of $\mathrm{GL}_2$ over totally real fields \cite{barnet2011sato}. Serre vastly generalized this to motives \cite{SerreMotivicGalois}, conjecturing that their normalized local $L$-factors are equidistributed according to the Haar measure on the motivic Sato-Tate group. While completely open for arbitrary motives, recent applications of potential automorphy lifting theorems have proven the generalized conjecture for certain non-generic abelian surfaces defined over totally real fields \cite{boxer2019abelian}.

\subsection{Results}
Michel \cite{Michel1995} established the following theorem on the second moment of the trace of Frobenius, $\mathcal{A}_2(p)$: 
\begin{theorem}[Michel]
\begin{align}
\sum_{t(p)}a_t(p)^2\ = \ p^2+O\left(p^{3/2}\right)
\end{align}
\end{theorem}

Following his work, our approach takes a canonical view toward the second moment by splitting it into the sum of terms indexed by half-powers of $p$. These terms are derived from the $\ell$-adic cohomology of a naturally constructed sheaf. In particular, as \cite{Michel1995} proved, the second moment splits into the main $p^2$ term, a $p^{3/2}$ term arising from the first cohomology over $\P^1$, and terms of $p$ or lower order arising from the singular fibers.
\begin{remark}
    The ``$p^{3/2}$ term'' here refers to the term arising from the first cohomology of $\P_1$. The $p$ term refers to order $p$ contributions arising from singular fibers. It may be possible that in some cases, the coefficient of the $p^{3/2}$ term is on the order of $O(p^{-1/2})$, leading to an additional $p$ term in the formula. We do not consider such a term to be a ``$p$ term'' as it arises from the cohomology part.
\end{remark}
Thus we refine the bias conjecture to the following statement, with the same conditions.

\begin{conjecture}
Let $\Cc_{w,p}$ be the coefficient of $p^{w/2}$ in the cohomological decomposition of the second moment. For the largest $w<4$ where $\Cc_{w,p}$ does not average to $0$, $\Cc_{w,p}$ instead averages to less than $0$ over the primes.
\end{conjecture}

This refinement was notably made in \cite{KN}, but, as they note, the decomposition was implicit in \cite{Michel1995}, from which we reformulated this version of the conjecture. 

We conditionally prove this conjecture for families with non-constant $j$-invariant. In particular, we have the following theorem.
\begin{theorem}
    Let $E_t$ denote a family of one-parameter elliptic curves over $\Q$ with non-constant 
    $j$-invariant. The second moment of this family decomposes into various cohomological terms. Let $\mathcal{C}_p^3$ and $\mathcal{C}_p^2$ denote the coefficients of the $p^{3/2}$ term and the $p$ term, respectively, in the cohomological decomposition of the second moment. Assuming the Sato-Tate conjecture holds, $\mathcal{C}_p^3$ will average to 0 in the limit. Additionally, $\mathcal{C}_p^2$ averages to a strictly negative value in the limit. 
\end{theorem}

In Section \ref{Proof}, we apply the Sato-Tate conjecture to the cohomological part that yields the $p^{3/2}$ term and use results from representation theory and motive theory to show, correspondingly, that $\Cc_p^3$ averages to 0 in the limit. Then, we find that the $p$ term arises from parts associated to ``bad'' points. In particular, terms arising from the contribution of the fibers of additive reduction in the cohomological decomposition contribute to $\Cc_p^2$. We show that this contribution goes to a non-positive value. Then, the only contribution to the $p$ term in the limit is given by $-B_pp$ where $B_p$ is the number of fibers of bad reduction. We show that  $B_p$ averages to a strictly positive value, conditionally proving the bias conjecture.

\section{Proof}\label{Proof}

\subsection{Preliminaries}
We begin by following the framework in \cite{Michel1995}. Fix a prime $\ell$ and define our one-parameter elliptic curve family over $\Z$, that is, the scheme over $\Z$ given by
\begin{align}
V:y^2+a_1(T)xy+a_3(T)y\ =\ x^3+a_2(T)x^2+a_4(T)x+a_6(T)
\end{align}
which is equipped with the projection morphism
\begin{align}
    f:V\to \A^1_{1/6\ell \Z}\ = \spec \left(\f{1}{6\ell}\Z[t] \right) 
\end{align}
sending $(x,y,t)\to T$.

To study the fiber $f^{-1}(t)$ for some $t$ and to describe this situation, we consider the sheaf \begin{align}
    \Ff\ :=\ R^1f_!\Q_\ell.
\end{align}
Note that this returns $\Ff_t=H^1(E_t)$. Since this sheaf is smooth on a dense open subscheme, we can also find an $N\in \N$ such that $\Ff$ is smooth of rank 2 on
\begin{align}
    U\ :=\ \spec\left(\Z\left[T,\f{1}{6\ell N\De^1(T)}\right]\right)
\end{align}
where $\De^1(T)=\De(T)/\gcd(\De(T),\De'(T))$ and $\De(T)$ is the discriminant.

Note that \begin{align}
    \tr(\frob_p,t|\Ff_p)\ =\ a_t(p).
\end{align}

Recall that
\begin{align}
    \Aa_2(p)\ :=\ \sum_{t\in \F_p}(\tr(\frob_p|H^1(E_t)))^2.
\end{align}

For non-singular points, we can use the identity \begin{align}
    \tr(\frob_p)^2\ =\ \tr(\sym^2(\frob_p))+\det(\frob_p).
\end{align}
Since $\det\frob_p=p$, we have
\begin{align}
    \tr(\frob_p)^2\ =\ \tr(\sym^2(\frob))+p.
\end{align}

In this context, $\tr(\sym^2(\frob))$ means the trace of the induced action of $\frob$ on $\sym^2H^1(E_t)$.

\subsection{Bounds on cohomology}

The Grothendieck-Lefschetz trace formula applied to $\sym^2\Ff_p$ gives us the equation
\begin{align}
    \sum_{t\in U(\F_p)}\tr(\frob_{p,t}\mid\sym^2\Ff_p)\ =\ \sum_{i=0}^2 (-1)^i \tr(\frob_p\mid H^i_c(\overline{U_p},\sym^2\Ff_p)).
\end{align}
Here, $\overline{U_p}\ :=\ U\otimes \overline{\F_p}$.
Note that this derivation is given in \cite{Michel1995} on page 133 for a similar setup.

Following Michel's proof of his Propositions 1.1 and 1.2 in \cite{Michel1995}, we can show $H^0$ and $H^2$ are trivial. On pages 133 and 134, \cite{Michel1995} explicitly states that 
\begin{align}
    H^0_c(\overline{U_p}, \sym^2\Ff_p))\ =\ 0
\end{align}
and \begin{align}
    H^2_c(\overline{U_p}, \sym^2\Ff_p))\ =\ 0.
\end{align}
It is important to note that the non-constant $j$-invariant condition is used in \cite{Michel1995} to show the second cohomology group is trivial. Without this condition, $H^2_c$ can be non-trivial, leading to additional terms for the second moment. This condition is crucial for the decomposition of the second moment into cohomological terms.

Then, following the proof of Proposition 1.2 in \cite{Michel1995}, letting $\Gg=\sym^2\Ff$, we get the short exact sequence on $\P_1\otimes \F_q$ given by
\begin{align}
    0 \to j_!\Gg\to j_*\Gg\to \Gg^{I_\infty}\otimes \bigoplus_{x\in \F_q,\De(x)=0}\Gg^{I_{\overline{x}}}\to 0.
\end{align}\label{SES1}

Here, $j:U\xhookrightarrow{}\P_1$. Note here that $I_{\overline{x}}$ is the inertia group of a place $\overline{x}$ over $x$. This in turn induces a long exact sequence on cohomology which simplifies to
\begin{align}
0\rightarrow\Gg^{I_{\infty}}\bigoplus_{x\in \F_{q},\Delta(x)=0}\Gg^{I_{x}}\rightarrow H_{c}^{1}(U_{q}\otimes\overline{\F}_{q}\vert{}\Gg)\rightarrow H^{1}(\mathbb{P}^{1}\otimes\overline{\F}_{q}\vert{}j_{*}\Gg)\rightarrow0
\end{align} where the kernel of this sequence is the $0^{\text{th}}$ cohomology group of the cokernel of our original short exact sequence.
Again, see \cite{Michel1995} for details.

By the proof of Lemma 4.1 in \cite{Michel1995}, we see that $\text{Sym}^2\mathcal F$ is pure of weight $2$ so $H^{1}(\mathbb{P}^{1}\otimes\overline{\F}_{q}\vert{}j_{*}\Gg)$ is pure of weight 3 and $\Gg^{I_{\infty}}\oplus_{x\in \F_{q},\Delta(x)=0}\Gg^{I_{x}}$ is mixed of weight $\leq 2$.
To clarify, for constructible $\ell$-adic sheaves, if we have a weight $n$ subspace, $V$, then that means that $\frob$ acts with eigenvalues of magnitude $p^{n/2}$ on $V$.

For the weight 3 part (which makes up $\Cc_p^3$), we have to treat $H^1(\P_1,\sym^2\Ff)$ as a motive. For this proof, a motive is a generalization of the cohomology space. Here, we use the theory of motives only because many results in the literature are phrased in terms of motives. See \cite{Milne2012Motives} for details on motives. 
\begin{remark}
    The space, $H^1(\P^1,\sym^2\Ff)$, is pure of weight 3. 
\end{remark}

\subsection{Bound on $H^1$ via Sato-Tate}


The Sato-Tate group of a motive is a compact Lie group, $G_{ST}$, into which we can embed a representation that maps Frobenius elements to conjugacy classes.\footnote{As a side note, the cohomology space, $H^1(\P_1,j_*\sym^2\Ff)$ carries a symplectic pairing so has even dimension. Thus, we can say $G\subseteq USp(g)$.} The generalized Sato-Tate conjecture (see \cite{ffk12, LMFDB_ST_Group}) implies that the values of conjugacy classes of the Sato-Tate group are equidistributed with respect to the Haar measure.
We apply the generalized Sato-Tate conjecture to the motive $H^1(\P_1, j_*\sym^2\Ff)\otimes \overline{\Q}_\ell(3/2)$ as follows:
\begin{align}
    \lim_{N\to \infty}\f{1}{\pi(N)}\sum_{p\leq N}\tr(\frob_p\mid H^1(\P_1,j_*\sym^2\Ff)\otimes \overline{\Q}_\ell(3/2))\ =\ \int_{G_{ST}}\tr(g)\ d\mu
\end{align}
where $G_{ST}$ is the Sato-Tate group of $H^1(\P_1,j_*\sym^2\Ff)\otimes \overline{\Q}(3/2)$.

Tensoring by $\overline{\Q}(3/2)$ is referred to as a Tate ``twist''. Recall that the eigenvalues had magnitude $p^{3/2}$. The twist divides the eigenvalues by $p^{3/2}$, ensuring the eigenvalues of the twisted space lie on the unit circle. See \cite{Michel1995, weilii} for examples of this approach. This ensures that the eigenvalues are bounded, so we can express Sato-Tate with the equation above.

We prove that the right-hand integral vanishes with the following lemma.

\begin{lemma}
For a given pure motive $X$ of odd weight, all irreducible components of the standard representation of the Sato-Tate group of $X$ are orthogonal to the trivial representation.
\end{lemma}

We sketch the proof. Suppose the standard representation on $G_{ST}$ is not irreducible. We split it into the sum of irreducibles. Suppose one of these irreducibles is trivial. This induces a one-dimensional representation of $\gal(\overline{\Q}/\Q)$ which is a one-dimensional submotive of $X$. Note that this submotive is simple (as it has rank 1). Let us denote it by $X_0$.


\begin{restatable}{lemma}{derham}\label{derham}
    The Galois action of $\gal(\overline{\Q}/\Q)$ on $H^1(\P_1,j_*\sym^2\Ff)$ is de Rham.
\end{restatable}
See Section \ref{de Rham} for a proof.

\begin{remark}
    We provide a brief intuitive explanation for the previous lemma. This submotive arises from a geometric construction, so is locally algebraic and in turn is de Rham. More precisely, $H^1(\P_1,j_*\sym^2\Ff$) is a mixed motive because it is constructed from the six functors (under which the category of constructible motives is closed). See \cite{Cisinski_2019} for details. Thus, $H^1(\P_1,j_*\sym^2\Ff)$ is constructed from pure motives. The $\ell$-adic realization of pure motives is de Rham by \cite{taylor}. Thus, since  $H^1(\P_1,j_*\sym^2\Ff)$ is constructed from pure motives, its $\ell$-adic representation is also a de Rham representation because the category of de Rham representations is closed under subquotients and extensions.
\end{remark}

We apply Conjecture 1.3 in \cite{taylor}, which he states is proven for 1-dimensional representations. The conjecture states that for de Rham representations, the weights must always be even. However, the subrepresentation must inherit the odd weight of the motive, a contradiction.

Since the representation of the Galois group has no 1-dimensional subrepresentations, we see that the same conclusion holds for the representation of the Sato-Tate group on the normalized cohomology group. Thus, by Schur orthogonality relations, the trace of the representation of the Sato-Tate group pairs to 0 with the trivial character, and the integral evaluates to 0. This shows that $\Cc_p^3$ averages to 0 in the limit.

\subsection{Multiplicative Bad Reduction}

We have shown that assuming Sato-Tate, the $p^{3/2}$ terms average to zero. Thus, we only need to consider the lower-order terms. Our results on the $p$ term do not require Sato-Tate. In the following section, we apply standard results from representation theory, Galois theory, and algebraic number theory. See \cite{lang1994} and \cite{artin2010} for details.

Let us consider $\Gg^{I_{\overline{x}}}$ for a point $x$ with $\De(x)$ which is degenerate of multiplicative type.
As per \cite{Michel1995}, the inertia group acts on $\Ff$ by the matrix
\begin{align}
    \gamma \ \mapsto\ \begin{bmatrix}
    1 & t_l(\gamma)\\
    0 & 1
\end{bmatrix}
\end{align} 
where $t_l$ is the surjection $I_x\to I^{\text{tame}}_x.$

Let $e_1,e_2$ denote the chosen basis for $\Ff$. For $\Gg=\sym^2\Ff$, the basis is given by $e_1^2,e_1e_2,e_2^2$.
We see that $e_1^2$ is the only eigenvector, so $\Gg^{I_{\overline{x}}}$ has dimension 1. 

Note that the inertia group is normal, so Frobenius normalizes the inertia group. Thus, if $x$ is an eigenvector for the action of the inertia group, then $\frob_p x$ is also an eigenvector. Letting $\tau$ denote a generator for the inertia group, it is known that $\tau \frob_p=\frob_p \tau^p$. Since $\frob_pe_1$ is a scalar multiple of $e_1$, a quick computation shows that $\frob_p$ acts with eigenvalue $\pm 1$ on $e_1$. Thus, $\Gg^{I_{\overline{x}}}$ has weight 0. In particular, we have shown that the trace over the terms of $\Gg^{I_{\infty}}\bigoplus_{x\in \F_{q},\Delta(x)=0}\Gg^{I_{x}}$ coming from degeneracies of multiplicative type is exactly the number of such degeneracies. Let us denote this number $M_p$ and note that $M_p$ is bounded by 
\begin{align}
    |\Set{x\in \F_p|\De(x)=0}|
\end{align}
which is independent of $p$. Thus, this gives us an $O(1)$ contribution.

\subsection{Additive Bad Reduction}
Now, for the points which are degenerate of additive type, a similar calculation yields that the eigenvalues are exactly $\pm p$. As a verification, we can obtain these values from the Euler factors which are computed in \cite{dummigan2009}.

Note that in \cite{dummigan2009}, for $p\geq 5$, the inertia group is cyclic of order $d\in\{2, 3,4,6\}$. We see that $d$ is independent of $p$ and only depends on the order of the discriminant. 

First, let us consider $d\in \{3,4,6\}$. Here, $\Gg^{I_x}$ is of dimension 1 as the kernel of the monodromy matrix has dimension one. This follows the dimension calculation from above. See \cite{stiller1981monodromy} for details on the matrices.
Now, \cite{dummigan2009} states that if $p=1\pmod{d}$, then the group $H$, which is generated by the Frobenius element and a generator of the inertia group, is abelian. This corresponds to the case where the eigenvalue is $p$. The other case must be $p=-1 \pmod{d}$ since $d\in \{3,4,6\}$. Thus, the coefficient of the $p$ term will be $A_p$ which is the number of fibers of additive reduction, possibly multiplied by -1. The sign is positive if and only if $p=1\pmod{d}$ for some $d\in \{3,4,6\}$. 

Recall that the number of fibers of additive reduction in $\F_p$ is given by the number of roots of a polynomial $P(T)$ over $\F_p$. This polynomial is the radical of $\gcd(A(T),B(T)$
\cite{silverman2009}.
Let $(\si, V)$ denote the permutation representation of $\gal(K,\Q)$ acting on the roots of $P(t)$ where $K$ is the splitting field of $P(T)$. Let $(\chi_d,W)$ denote the non-trivial 1-dimensional representation of $\gal(\Q(\zeta_d)/\Q)$. 
Now, we see that the number of roots of $P(T)$ over $\F_p$ is given by the trace of $\frob_p$ acting on $V$. 
The trace of $\frob_p$ acting on $W$ is determined by the value of $p\pmod{d}$. Thus, $ A_p$ is given by multiplying the traces of these two representations. 
Define 
\begin{align}
    L\ =\ K(\zeta_d)
\end{align}
to be the compositum field. Then, note that $\gal(K/\Q)$ and $\gal(\Q(\zeta_d)/\Q)$ are quotients of $G:=\gal(L/\Q)$. Thus, we can extend $\rho$ and $\si$ to be representations of $G$. We can define a new representation $V\otimes W$ giving us
\begin{align}
    A_p \ =\ \tr(\frob_p|W)\tr(\frob_p|V)\ =\ \tr(\frob_p|V\otimes W)\ =\ \chi_{\rho\otimes \si}(\frob_p).
\end{align}
Now, we want to find the average value of this quantity. Recall that $K$ and $\Q(\zeta_d)$ were finite extensions, so $L/\Q$ is also a finite extension. Thus, $G$ is finite, and we can apply the Chebotarev density theorem \cite{lang1994}. This states that for a conjugacy class, $C\subseteq G$, the proportion of all primes such that $\frob_p\in C$ is exactly 
$|C||G|$.
Note that the trace of a representation is constant on conjugacy classes. Letting $\pi_C(x)$ denote the number of primes $p\leq x$ where $\frob_p\in C$, we have
\begin{align}
    \lim_{x\to\infty}\f{1}{\pi(x)}\sum_{p\leq x}\tr(\frob_p|V\otimes W))\ =\ \lim_{x\to\infty}\sum_C\f{\pi_C(x)}{\pi(x)}\tr((\rho\otimes\si)(g_c))
\end{align}
where $g_c$ is some element in $C$. Thus, we simplify to
\begin{align}
    \lim_{x\to\infty}\sum_C\f{\pi_C(x)}{\pi(x)}\tr((\rho\otimes\si)(g_c))&\ =\ \sum_C \left(\lim_{x\to\infty}\f{\pi_c(x)}{\pi(x)}\right)\tr((\rho\otimes\si)(g_c))\nonumber\\
    &\ =\ \sum_C \f{|C|}{|G|}\tr((\rho\otimes\si)(g_c))\nonumber\\
    &\ =\ \sum_{g\in G} \tr((\rho\otimes\si)(g))\nonumber\\
    &\ =\ \la \chi_{\rho}\otimes\chi_{\si},\chi_{\text{triv}}\ra\nonumber\\
    &\ =\  \la \chi_\si,\chi_\rho\ra.
\end{align}
We see that this inner product must always be non-negative. Note that in most cases, $K$ and $\Q(\zeta_d)$ will be linearly disjoint, so the inner product will be zero. In general, we can conclude that $A_p$ averages to a non-negative value. 

Now, we consider the $d=2$ case. 
Recall that an additive bad reduction at $t_0$ means that $A(t_0)=B(t_0)=0 \pmod p$. There are at most $\deg(\De)$ such points and two cases where this happens. 

In the first case, $A(t_0)=B(t_0)=0$ over $\Q$. Then, the fiber at $t_0$ is of additive bad reduction for all primes (sufficiently large).

In the second case, $A(t_0)=C_0$ and $B(t_0)=D_0$, and without loss of generality, $C_0, D_0\neq 0$ over $\Q$. Then, $A(t_0)=B(t_0)=0\pmod p$ only at primes where $p\mid \gcd(C_0,D_0)$, that is, at finitely many primes. We can ignore these contributions in the average since they only occur for finitely many primes.

Thus, without loss of generality, we may assume all fibers of additive reduction with $d=2$ occur at infinitely many primes.\footnote{This is the case where the Kodaira type is $I_0^*$.}
Here, the monodromy action is given by the matrix \begin{align}
    \begin{bmatrix}
        -1 & 0\\
        0 & -1
    \end{bmatrix};
\end{align}
see \cite{stiller1981monodromy}. For the symmetric square, the inertia group acts as $(-I)^2=I$.
Thus, we see the entire space is an eigenspace and $\Gg^{I_{t_0}}=\Gg_{\eta}$, which denotes the generic local fiber. Now, we need to compute the average of
\begin{align}
    \tr(\frob_p|\Gg_{\eta}).
\end{align}

It is known that for additive singularities with $d=2$, we need a quadratic extension to resolve the singularity \cite{silverman}. Let us consider the formal neighborhood of $t_0$ by considering the local field $K:=\Q((t-t_0))$. Adjoining $\al:=\sqrt{t-t_0}$ gives us an isomorphism between the local generic fiber, $E_{k}$, and some smooth curve $\Ee$ over $K$ which has good reduction at $t_0$. 
Note that this isomorphism is over $K(\al)$.
Every element of $\gal(\overline{K}/K)$ maps $\al \mapsto \pm \al$. This defines a character $\chi:\gal(\overline{K}/K)\to \{\pm1\}$ that maps a Galois element to the parity of its action on $\al$. Note that $\Ff_{\eta}=H^1(E_K)$ is a representation of $\gal(\overline{K}/K)$, call this representation $\si_{0}$. 

Let $\si$ denote the representation acting on $H^1(\Ee)$, and we see that 
\begin{align}
    \si\ =\ \si_0\otimes \chi.
\end{align}
Thus, we have 
\begin{align}
    \sym^2(H^1(\Ee))\ =\ \sym^2(\Ff_{\eta})\otimes \chi^2\ =\ \sym^2(\Ff_{\eta})\ =\ (\sym^2\Ff)_\eta.
\end{align}
We define $E_{0}$ as the special fiber of $\Ee$ at $t_0$ and note that it is non-singular since $\Ee$ had good reduction at $t_0$. By the Neron-Ogg-Shafarevich condition \cite{silverman}, we see that $I_{t_0}$ acts trivially on $H^1(\Ee)$. Thus, the action of $\gal(\bar{K}/K)$ factors through an action of $\gal(\bar{\Q}/\Q)$.
This gives us that $\gal(\bar{\Q}/\Q)$ acts the same on both $H^1(\Ee)$ and $H^1(E_0)$. Thus, they are isomorphic representations, and 
\begin{align}
    \sym^2(\Ff)_\eta\ =\ \sym^2(H^1(\Ee))\ =\ \sym^2(H^1(E_0)).
\end{align}
Thus, it suffices to compute the average of
\begin{align}
    \tr(\frob_p|(\sym^2(H^1(E_0))).
\end{align}

Here, $H^1(E_0)$ is a 2-dimensional vector space since it is non-singular, so we can apply 
\begin{align}
    \tr(\frob_p|(\sym^2(H^1(E_0)))\ =\ \tr(\frob_p|H^1(E_0))^2-p\ =\ a_{E_0}(p)^2-p.
\end{align}
By \cite{serre2012lectures}, $a_{E_0}(p)^2/p$ averages to 1 in the limit. Thus, we see that $\tr(\frob_p|\Gg^{I_{t_0}})$ must grow on the order of at least $p$. We already know it cannot grow faster than $p$ since we are in the weight 2 part. 
We have also shown that its coefficient, $a_{E_0}(p)^2/p-1$, averages to 0. 
Thus, we see that this $\Gg^{I_{t_0}}$ contributes a $p$ term whose coefficient averages to 0.

Thus, the total coefficient of the $p$ term, which arises from $\tr(\text{Frob}_{p, t}\mid \sym^2\Ff)$, is, on average, zero.

\subsection{Final Computation}

We have computed the trace at the non-singular fibers, i.e., where $\De(t)\neq0$. Now, let us count the contribution from the points where $\De(t)=0$.

Let us use $B_p$ to denote the number of these points in $\F_p$. At these points, $|a_t(p)|\leq 1$, so we get a contribution of at most $B_p$, i.e., an $O(1)$ contribution.

Finally, we have
\begin{align}
    \Aa_2(p)&\ =\ \sum_{t\in \F_p}\tr(\frob_p|\Ff))^2\nonumber \\
    &\ =\  \sum_{t\in U(\F_p)}(\tr(\frob_p|\Ff))^2+\sum_{\De(t)=0}(\tr(\frob_p|\Ff))^2\nonumber \\
    &\ =\  \sum_{t\in U(\F_p)}\left(\tr(\frob_p|\sym^2\Ff)+p\right)+O(1)\nonumber \\
    &\ =\  -C_pp^{3/2}-D_pp+p(p-B_p) +O(1)
\end{align}
where $D_p$ is the term arising from the fibers of additive reduction.
We have shown that $C_p$ averages to 0 in the limit, and that $D_p$ averages to a non-negative value. Additionally, $\De(t)$ has at least one root on average over the primes since the average number of roots is given by the number of irreducible factors of $\De(t)$. \cite{serre2012lectures}. Thus, the average of $B_p$ is $\geq 1$, so the $-B_p$ term averages to a negative value. Thus, the coefficient of the $p$ term averages to a negative value in the limit.

\subsection{The submotive is a de Rham representation}\label{de Rham}
\derham*
\begin{proof}
We aim to realize our motive $H^1(\P_1,j_*\sym^2\Ff)$ as a representation constructed by subquotient and extension of the $\ell$-adic cohomology of some smooth proper variety.

First, we recall that $H^1(\P_1,j_*\sym^2\Ff)$ is a quotient of $H^1_c(U,\sym^2\Ff)$ (see equation \eqref{SES1}).  Note here that we consider these as spaces over $\overline{\Q}$. Recall that $\Ff=R^1f_!\Q_\ell$ where $f:X\to U$ and $X$ is the elliptic surface over $\Z$. We let $g: Y:=X\times_U X\to U$ be the fiber product. Note that $\Ff\otimes \Ff$ is a direct summand of $R^2g_!\Q_\ell$ and $\sym^2\Ff$ is a quotient of $\Ff\otimes \Ff$. Thus, we can let $\Gg=R^2g_!\Q_\ell$ and then we see that $H^1_c(U,\sym^2\Ff)$ is a direct summand of $H^1_c(U, \Gg)$. Note here that exact sequences split since we are working over $\Q_\ell$.

Next, we apply the Leray spectral sequence for compactly supported pushforward to $g: Y\to U$. We have
\begin{align}
    E_2^{p,q}\ = \ H^p_c(U,R^qg_!\Q_\ell)
\end{align}
which converges to $H^{p+q}_c(Y,\Q_\ell)$.
Note that $E_{2}^{-1,3}=H^{-1}_c(U,R^3g_!\Q_\ell)$ vanishes and \begin{equation}
    E_{2}^{3,1}=H^{3}_c(U,R^2g_!\Q_\ell)
\end{equation} also vanishes since $U$ is 1-dimensional. Considering the differential $d_r:E_r^{p,q}\to E_r^{p+r,q-r+1}$, we see that the terms before and after $E_2^{1,2}$ are 0. Thus, $E_3^{1,2}=E_2^{1,2}$.
However, we also see that for $r\geq 2$, we have that $E_2^{1-r,3+r}$ and $E_2^{1+r,3-r}$ vanish for the same reasons as above. Thus, the terms before and after $E_3^{1,2}$ are 0 and $E_3^{1,2}=E_4^{1,2}$. We proceed inductively and find that \begin{align}
    E_\infty^{1,2}\ =\ E_2^{1,2}.
\end{align}
Thus, $E_2^{1,2}=H^1_c(U, \Gg)$ is a graded piece, i.e., a subquotient, of $H^3(Y,\Q_\ell)$.

Now, since we are in characteristic zero, we apply Hironaka's theorem (see page 89 of \cite{peters2008mixed}) which guarantees that we can find a smooth compactification, i.e., a smooth proper variety $\overline{Y}$ of $Y$ such that $D:=\overline{Y}\setminus Y$ is a normal crossings divisor. Let $D_1,\dots, D_n$ denote the irreducible components of $D$. We consider the weight spectral sequence 
\begin{align}
    E_1^{-r,k+r}\ =\ \bigoplus_{|I|=r}H^{k-r}(D_{I},\Q_\ell)
\end{align}
where $D_I=\cap_{i\in I}D_i$. This converges to $H_c^k(Y,\Q_\ell)$. Note that the $D_I$ are smooth and proper as they are closed subsets of $\overline{Y}$. Thus, $H_c^3(Y,\Q_\ell)$ is constructed through subquotients and extensions of cohomology groups of smooth, proper varieties over $\Q$.

Now, Faltings and Tsuji have proven that the $\ell$-adic cohomology of smooth proper varieties over $\Q_\ell$ are de Rham representations of $\gal(\overline{\Q_\ell}/\Q_\ell)$ (see Theorem 5.33 in \cite{fontaine_ouyang_2008}). Note the category of de Rham representations is closed under subquotients and extensions (see Theorems 2.13 and the beginning of Section 5.2.5 in \cite{fontaine_ouyang_2008}). Thus, since $H^1(\P_1,j_*\sym^2\Ff)$ can be obtained from de Rham representations via subquotients and extensions, we see that it is a de Rham representation.
\end{proof}

\section{Examples}
We verify that the results from the proof agree with some known computations.

\subsection{The family $y^2=x^3+x+T^3$}
\,

The family, $E_t:y^2=x^3+x+T^3$, is studied in \cite{batterman2024applications}. Note here that the $j$-invariant of $E_t$ is non-constant and is given by $j = 1728 \left( \frac{4}{4 + 27T^6} \right) = \frac{6912}{4 + 27T^6}$. They proved that for $p=2 \pmod 3$, the second moment is 
\begin{align}
    \Aa_2(p)\ =\ p^2+p.
\end{align}
This family was identified as a potential counterexample for the bias conjecture since the coefficient of the $p$ term is positive for $p=2 \mod 3$. In particular, it is positive for half the primes.

Here, we verify that this aligns with our results, showing that it is indeed possible for the coefficient of $p$ to be positive for a positive percentage of $p$.

The discriminant is given by $\De(t)=16(4+27t^6)$. Let us find when this is equal to 0 for $p=2\pmod 3$. Substituting $s=t^2$, we have
\begin{align}
    \De(t)=0\implies s^3=-\f{4}{27}.
\end{align}
Since $p=2\pmod 3$, unique cube roots exist, and we have some $s_0$ that satisfies the equation. Since $s_0$ is a square, $s_0^3$ must be a square. We see that 
\begin{align}
    -\f{4}{27}\ =\ -3\cdot\frac{2^2}{9^2}
\end{align}
so $-3$ must be a square. However, the Legendre symbol
\begin{align}
    \left(\f{-3}{p}\right )
\end{align}
is equal to 1 if and only if $p=1\pmod 3$. Thus, $s_0$ has no squares, and the discriminant has no roots for $p=2\pmod 3$.

For $p=1\pmod 3$, we note that 
\begin{align}
    -\f{4}{27}\ =\ -3\cdot \frac{2^2}{9^2}
\end{align}
is always a square since $-3$ is always a square for $p = 1\pmod 3$.
Then, we need to check if 
\begin{align}
    -\f{4}{27}\ =\ -4\cdot\frac{1^3}{3^3}
\end{align}
is a cube. Since $-1$ is a cube, we see that this is equivalent to checking if $4$ is a cube, i.e., if $2$ is a cube. It is known that this happens for a third of the primes $p=1\pmod 3$. Since $\De'(t)=0$ only at $t=0$, we see that we do not have repeated roots. Thus, we have no roots if $2$ is not a cube, and we have six roots if $2$ is a cube.

Next, we check if $\infty$ is of additive reduction. If we apply the transformation $s=1/t$, we get
\begin{align}
    A'(s)\ =\ A(1/t)s^{4k}\notag\\
    B'(s)\ =\ B(1/t)s^{6k}
\end{align}
for the minimal $k$ such that these values are rational. Letting $k=1$, we get 
\begin{align}
    E_s:y^2\ =\ x^3+s^4x+s^3.
\end{align}
At $s=0$, we have that $A'(0)=B'(0)$, so $t=\infty$ is of additive reduction for all $p$. We see that the minimal order of the determinant is 6, so this yields the $d=2$ case (see \cite{dummigan2009}).

We consider this curve over $\Q((s))\sqrt{s}$. We make the substitution $X=\al^2x=sx$ and $Y=\al^3y=s\sqrt{s}y$ giving us
\begin{align}
    Y^2\ =\ X^3+s^2X+1.
\end{align}
We plug in $s=0$ to get the curve
\begin{align}
    Y^2\ =\ X^3+1.
\end{align}

Let us find the trace. Firstly, for $p=2 \pmod 3$, we see that $x\mapsto x^3$ is an automorphism, so $X\to X^3+1$ is a bijection. We see that we have $(p+1)/2$ quadratic residues in $\F_p$, we have $2(p+1)/2$ solutions. Thus, 
\begin{align}
    a_{E_0}(p)\ =\ p+1-(p+1)\ =\ 0.
\end{align}

Then, for the $p=1\pmod 3$ case, we follow Section 18.3 from \cite{ireland1990classical}. Following their derivation, for $p=1\pmod 3$, we get
\begin{align}
    a_{E_0}(p)\ =\ \left(\f{4}{\pi}\right)(\pi+\overline{\pi})\ =\ \pi+\overline{\pi}
\end{align}
where $\pi\bar{\pi}=p$ for some $\pi\in \Z[\om]$ with $\pi=2\pmod 3$. We can write $\pi=c+b\om$ with $c^2+3d^2=p$. Thus, $c\neq 0$ and the squared trace is given by 
\begin{align}
    a_{E_0}(p)^2\ =\ 4c^2.
\end{align}
Now, this value can get arbitrarily close to 0 since the normalized values of $\pi$ are equidistributed on the circle (see \cite{serre1968abelian}). However, we see that the average value of $4\cos^2\theta=4c^2$ is exactly 2.

There are no fibers of bad reduction away from infinity, so $B_p=0$. Thus, the coefficient of the $p$ term in the second moment is given as follows:
\begin{align}
    \begin{cases}
        +1 & p=2\pmod 3\\[2mm]
        1-4\cos^2\theta_p & p=1\pmod 3, \, \left(\f{2}{p}\right)_3=-1\\[3mm]
        1-4\cos^2\theta_p-6 & p=1\pmod 3, \, \left(\f{2}{p}\right)_3=1\\
    \end{cases}
\end{align}
where $\theta_p$ is equidistributed in $[0,\pi]$. Recall that $2$ is a cube mod 3 for a third of the primes $p=1\pmod 3$. Thus, we have
\begin{align}
    \lim_{x\to\infty}\f{1}{\pi(x)}\sum_{p\leq x}\f{\Aa_2(p)-p^2-C_pp^{3/2}}{p}\ =\ 
    \left(1-\f{1}{2\pi}\int_{0}^{\pi}4\cos^2\theta\ d\theta\right)-\left(\f{1}{6}\cdot 6+\f{5}{6}\cdot 0\right)\ =\ -1.
\end{align}
This shows that the average of the coefficient of the $p$ term is indeed negative, although the coefficient is $+1$ for half of the primes.
\subsection{The family $y^2=x^3+x^2+2T+1$}\,

In Weierstrass form, this family is given by 
\begin{align}
    E_t:y^2\ =\ x^3-\f{1}{3}x+\left(2t+\f{29}{27}\right)
\end{align}
and its discriminant is given by
\begin{align}
    \De(t)\ =\ -12(2t+1)(54t+31).
\end{align}
It is clear that the $j$-invariant in nonconstant. We see that there are two fibers of multiplicative reduction, $-1/2$ and $-31/54$. Additionally, $\infty$ is a fiber of additive reduction. We see that the discriminant has order 2, so $d=\f{12}{\gcd(12,2)}=6$. Recall that the Legendre symbol yields
\begin{align}
    \left(\f{-3}{p}\right)\ =\ \begin{cases}
        +1& p=1\pmod 3\\
        -1 & p=2 \pmod 3
    \end{cases}\ =\ \begin{cases}
        +1& p=1\pmod 6\\
        -1 & p=2 \pmod 6.
    \end{cases}
\end{align}
Thus, the coefficient of the $p$ term is given by
\begin{align}
   \f{ \Aa_2(p)-p^2-C_pp^{3/2}}{p
   }\ =\ -2-\left(\f{-3}{p}\right),
\end{align}
which agrees with the computation in previous work, such as \cite{Miller2002}\footnote{This work has a typo in the computation of this example. In particular, the $p$ is omitted in the $\left(\frac{-3}{p}\right)$ term.}.

\section{Future Work}
\subsection{ Notes on Sato-Tate}
Considering the $H^1$ space as a representation $\rho$ of $\gal(\overline{\Q}/\Q)$, applying Chebotarev's density theorem is difficult, because one must show the Galois representation factors through a finite group, or one must show that $\Set{g\in\gal(\bar{\Q}/\Q)|\chi_\rho(g)=x}$, for some trace $x$, has measure zero boundary. Deligne's equidistribution theorems do not apply here since we are applying Frobenius to a cohomology group. It may be possible to show that $H^1(\P_1,j_*\sym^2\Ff)$ is automorphic which is a case where Sato-Tate has been studied more carefully. 

\subsection{Constant $j$-invariant}

We see the iso-trivial families arise from one of the following \cite{conrad2005root}:
\begin{itemize}
    \item quadratic twists, 
    \item quartic twists when $j=1728$, and
    \item cubic or sextic twists when $j=0$.
\end{itemize}

This may split into further cases depending on whether or not the curve has complex multiplication. A similar proof method may apply.

In particular, for $j=0,1728$, the same proof method should work. The second cohomology should vanish since the base curve is twisted by a cubic, quartic, or sextic character $\chi$. When $\chi^2$ is non-trivial, \cite{Michel1995} notes that 
\begin{align}
    H^2(U,\sym^2\Ff))\ =\ H^2(U,\sym^2\Ff_0)\otimes \Ll_{\chi^2})
\end{align}
where $\Ff_0=R^1f_!\Q_\ell$ is the untwisted sheaf on $\P^1$ corresponding to the fibers of the generic curve $E\times_\Z \P^1$. Note that this is the untwisted $E(T)$. One must carefully check that the rest of the argument holds in this case. 

For $j\neq 0,1728$, our formulation of the bias conjecture is not well-defined since there is no cohomological decomposition of the trace in this case. To extend the proof to this case, one must formulate an analogous trace decomposition for curves with constant $j$. Here, it may be possible that $H^1$ will vanish or that $H^2$ will contribute terms to the second moment. This is supported by numerical data as we see that the coefficient of the $p^2$ may not always be one for families of constant $j$-invariant \cite{Miller2002}.

\section{Acknowledgments}
We would like to thank Professors Matija Kazalicki, Bartosz Naskr\k{e}cki, and Adam Logan for reading our proof and providing comments. We also thank Ritik Jain for pointing out an error in Section 2.5 in an earlier draft. This research was supported with funding from the National Science Foundation (grant DMS2341670), Texas A\&M University, the University of Chicago, the University of Michigan, and Williams College. We also acknowledge the use of generative AI tools for the purposes of finding references, conducting literature reviews, and identifying errors.

\newpage

\nocite{*}
\printbibliography

\end{document}